\documentclass[12pt]{article} 

\usepackage[utf8]{inputenc} 
\usepackage{pgf,tikz} 
\usetikzlibrary{backgrounds}
\usetikzlibrary{arrows}
\usepackage{geometry} 
\usepackage{mathtools}

\usepackage{graphicx} 
\usepackage[bb=boondox,bbscaled=.95,cal=boondoxo]{mathalfa}
\usepackage[all,cmtip]{xy}
\usepackage{booktabs} 
\usepackage{array} 
\usepackage{paralist} 
\usepackage{verbatim} 
\usepackage{subfig} 
\usetikzlibrary{arrows.meta, calc, chains, positioning}
\usepackage{amsfonts,amsmath,latexsym,amssymb} 
\usepackage{amsthm}
\usepackage{amsmath}            
\usepackage{mathptmx}

\usepackage{booktabs} 
\usepackage{array} 
\usepackage{paralist} 
\usepackage{verbatim} 
\usepackage{subfig} 
\usepackage{amsfonts,amsmath,latexsym,amssymb} 
\usepackage{amsthm}            
\usepackage{mathrsfs,upref}         
\usepackage{mathptmx}
\usepackage{fancyhdr} 
\usepackage[all,cmtip]{xy}
\usepackage{tikz-cd}
\usepackage{sectsty}
\allsectionsfont{\sffamily\mdseries\upshape} 
\sectionfont{\centering \sc \LARGE}
\subsectionfont{\centering \sc \large}
\usepackage[nottoc,notlof,notlot]{tocbibind} 
\usepackage[titles,subfigure]{tocloft} 

\newtheorem{thm}{\textsc{Theorem}}

\newtheorem{lem}[thm]{\textsc{Lemma}}

\newtheorem{prop}[thm]{\textsc{Proposition}}

\newtheorem*{prop*}{Proposition}
\newtheorem*{lem*}{\textsc{Lemma}}

\newtheorem*{rem}{Remark}
\newtheorem*{fact*}{\textsc{\textbf{Fact}}}

\theoremstyle{definition}
\newtheorem*{exa*}{Example}
\theoremstyle{definition}
\newtheorem*{ack}{Acknowledgement}
\newtheorem{defn}[thm]{Definition}
\newtheorem*{defn*}{Definition}
\theoremstyle{remark}

\newtheorem{thm*}[thm]{Theorem}

\DeclareMathOperator{\G}{\mathcal{Su}(3)}

\DeclareMathOperator{\Go}{\mathcal{Su}(3)_{\mathrm{obj}}}

\DeclareMathOperator{\im}{Im}

\DeclareMathOperator{\id}{Id}

\DeclareMathOperator{\Hh}{\mathrm{H}}

\DeclareMathOperator{\rk}{\mathfrak{r}}

\DeclareMathOperator{\F}{\mathcal{F}}
\DeclareMathOperator{\Fa}{\mathcal{F}_{\mathrm{arr}}}
\DeclareMathOperator{\Fo}{\mathcal{F}_{\mathrm{obj}}}
\title{Quantum $SU(3)$ as the $C^*$-algebra of a $2$-Graph}
\date{}
\author{Olof Giselsson}

\begin{document}
\maketitle
\abstract{We show that for $q\in (0,1),$ the $C^{*}$-algebra $SU_{q}(3)$ is isomorphic a rank $2$ graph $C^{*}$-algebra (in the sense of Pask and Kumjian). This graph is derived by passing the to the limit $q\to 0$ for a set of generators of $SU_{q}(3)$. Moreover, the isomorphism can be taken to be $\mathbb{T}^{2}$-equivariant with respect the right-action on $SU_{q}(3)$ and the gauge action coming from the $2$-graph.}

\section{introduction}
It's an observation that certain quantized algebras of continuous functions on homogeneous spaces can be given the structure of graph $C^{*}$-algebras. A well-known example are the Soibelman-Vaksman quantum odd spheres~\cite{vaso}. In particular, $SU_{q}(2)$ -- the algebra of quantized continuous functions on $SU(2)$ introduced by S.L. Woronowicz in~\cite{wor} -- is isomorphic to the graph $C^{*}$-algebra of the directed graph
$$
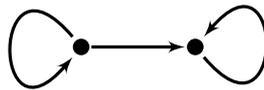
\begin{tikzpicture}[>=latex',node distance = 1.5cm]
\node (S2) {};
\node [left of = S2] (S1) {};
 \draw[->,very thick] (S1) to[out=130,in=230,looseness=16] node[above] {} (S1);
 \draw[->,very thick] (S2) to[out=-50,in=50,looseness=16] node[above] {} (S2);
\draw [->,very thick] (S1) to  (S2);
\foreach \n in {S1,S2}
  \node at (\n)[circle,fill,inner sep=2.2pt]{};
\end{tikzpicture}
$$
(Proposition 2.1 in~\cite{woj}). A natural question to ask would be if a similar result holds for $SU_{q}(n)$ when $n\geq 3$. However, the primitive spectrum of these $C^{*}$-algebras will topologically contain the space $\mathbb{T}^{n-1}$ and it's known that such $C^{*}$-algebras can not come from directed graphs (see the end of the introduction in~\cite{woj}). 
\\

In~\cite{hrank} David Pask and Alex Kumjian generalized the notion of a  $C^*$-algebra constructed from a directed graph, by introducing higher-rank graphs together with their associated Cuntz-Krieger algebras. In a graph of degree $k$ (or a $k$-graph), the usual dots-and-arrow presentation is replaced by a countable category $\mathcal{G}$ together with a functor $\mathfrak{r}:\mathcal{G}\to \mathbb{N}^{k}$ called the {\it rank functor} (the semi-group $\mathbb{N}^{k}$ is considered as a category with one element). In this context, a $1$-graph is equivalent to a category generated by a directed graph.
\\

In this paper we show that the compact quantum group $SU_{q}(3)$ (as a $C^{*}$-algebra) is isomorphic to a $C^{*}$-algebra coming from a $2$-graph $\mathfrak{r}:\G\to\mathbb{N}^{2}$, whose $2$-skeleton (i.e. the set of objects in $\G$ together with the pre-images of morphisms $\mathfrak{r}^{-1}(1,0)$ and $\mathfrak{r}^{-1}(0,1)$) can be visualized by the colored directed graph
$$
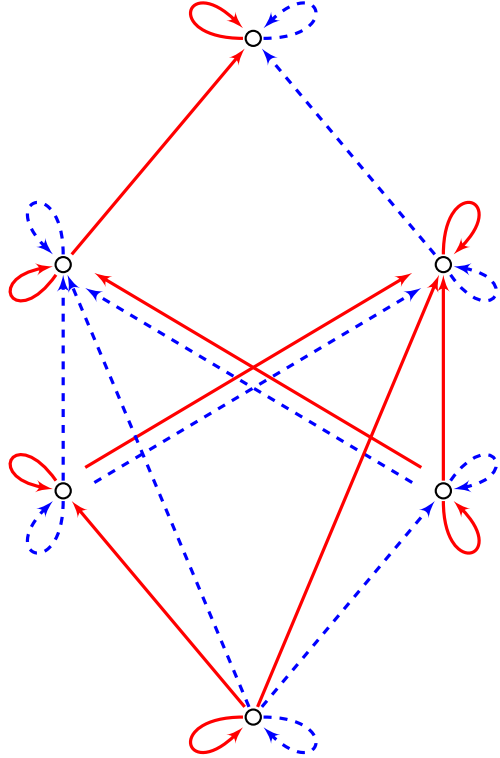
\begin{tikzpicture}[>=latex',node distance = 1cm]'
\node at (0,0) (CZ) {};
\node at (2.5,3) (CY) {};
\node at (-2.5,3) (BZ) {};
\node at (-2.5,6) (BX) {};
\node at (2.5,6) (AY) {};
\node at (0,9) (AX) {};
\draw [dashed,->,very thick, blue] (CZ) to (CY);
\draw [->,very thick, red] (CZ) to (BZ);
\draw [dashed,->,very thick, blue] (BZ) to (BX);

\draw [->,very thick,red] (2.5-0.35+0.06,3+0.21+0.1) to  (-2.5+0.35+0.06,6-0.21+0.1);
\draw [dashed,->,very thick,blue] (2.5-0.35-0.06,3+0.21-0.1) to (-2.5+0.35-0.06,6-0.21-0.1);

\draw [->,very thick,red] (-2.21,3.31) to  (2.09,5.89);
\draw [dashed,->,very thick,blue] (-2.09,3.11) to (2.21,5.69);

\draw [->,very thick,red] (CY) to (AY);
\draw [dashed,->,very thick, blue] (AY) to (AX);
\draw [->,very thick, red] (BX) to (AX);

 \draw[->,very thick,red] (CZ) to[out=180,in=225,looseness=20] (CZ);
 \draw[dashed,->,very thick, blue] (CZ) to[in=-45,out=0,looseness=20]  (CZ);

 \draw[dashed,->,very thick,blue] (BX) to[out=90,in=90+45,looseness=20] (BX);
 \draw[->,very thick, red] (BX) to[in=90+100,out=90+45+100,looseness=20] (BX);

 \draw[->,very thick,red] (BZ) to[out=125,in=125+45,looseness=20] (BZ);
 \draw[dashed,->,very thick, blue] (BZ) to[in=125+100,out=125+45+100,looseness=20] (BZ);

 \draw[->,very thick,red] (CY) to[out=90+180,in=90+45+180,looseness=20] (CY);
 \draw[dashed,->,very thick, blue] (CY) to[in=90+100+180,out=90+45+100+180,looseness=20] (CY);

 \draw[dashed,->,very thick,blue] (AY) to[out=125+180,in=125+45+180,looseness=20]  (AY);
 \draw[->,very thick, red] (AY) to[in=125+100+180,out=125+45+100+180,looseness=20] (AY);

 \draw[->,very thick,red] (AX) to[in=-225,out=-180,looseness=20] (AX);
 \draw[dashed,->,very thick, blue] (AX) to[out=0,in=45,looseness=20]  (AX);

\draw [dashed,->,very thick, blue] (CZ) to (BX);
\draw [->,very thick, red] (CZ) to (AY);

\foreach \n in {AX,AY,BX,BZ,CY,CZ}
  \draw[black, thick] (\n) circle (0.1 cm);
\end{tikzpicture}
$$
We get this graph by passing to the limit $q\to 0$ for an set of $SU_{q}(3)$-generators under the usual faithful Soibelman representation. This is similar to how the graph of $SU_{q}(2)$ is derived, where one passes to the limit $q\to 0$ on appropriate generators, and then show that the resulting limits still generates the same $C^{*}$-algebra. However, due to the additional dimension, some further steps are needed. Under the Soibelman representation, we realize $SU_{q}(3)$ as a $C^{*}$-subalgebra of $C^{*}(S)^{\otimes 3}\otimes C(\mathbb{T})\otimes C(\mathbb{T})$ (here $C^{*}(S)$ denotes the $C^{*}$-algebra generated by the forward shift $S\in \mathcal{B}(\ell^{2}(\mathbb{N})),$ i.e. the Toeplitz algebra), and the issue is that these $C^{*}$-subalgebras varies as a function of $q$, even though they are all isomorphic (see~\cite{gnagy} or~\cite{giselsson} for the latter claim). In order to then get an isomorphism with the limit at $0,$ one must first straighten out the field of $C^{*}$-algebras. Unfortunately, this method is non-constructive and based on a lifting lemma, so that in comparison with the $SU_{q}(2)$-case, it's much less clear how the $2$-graph "sits" inside of $SU_{q}(3).$ 
\\

The topic of this paper connects with some similar recent investigations: In~\cite{roma}, M. Matasso and R. Yuncken showed that for any compact semisimple Lie group $K,$ if one looks at an appropriate $*$-subalgebra of its quantized coordinate ring $\mathcal{O}_{q}[K]$ (of elements regular at $q=0$). Then under the faithful Soibelman representation, the limit $q\to 0$ exists for every element in this subalgebra, and the resulting $*$-algebra can be described as a Kumjian-Pask algebra coming from a higher rank graph (with the rank of the graph equal the rank of $K$). In particular, in the case of $SU(3)$, Matasso and Yuncken derive the same 2-graph as we do here. Whereas our 2-graph was the result of empirical investigations, they derive theirs from the theory of crystal bases. It thus seems reasonable to expect that the results here can be generalized to hold for every compact semisimple Lie group $K.$ 
\\
Another recent approach to understand the limit $q=0$ in the case of $SU_{q}(n)$ was given by M. Giri and A.K. Pal in~\cite{gir}. They give a description of $SU_{0}(n)$ in terms of generators and relations. They focus in particular on the case $SU_{0}(3)$ and classifies its irreducible representations.
\begin{ack}
This work was supported by the RCN (Research Council of Norway), grant 300837.
\end{ack}
\subsection{Higher Rank Graph $C^{*}$-algebras}
Graph $C^{*}$-algebras of higher rank is a generalization of graph $C^{*}$-algebras that was introduced by Kumjian and Pask in~\cite{hrank}. 
We give the definition:
For a small category $\mathcal{C},$ we let $\mathcal{C}_{\mathrm{obj}}$ denote the set of objects of $\mathcal{C}$ and $\mathcal{C}_{\mathrm{arr}}$ the set of arrow of $\mathcal{C}.$
\begin{defn}[{\sc Graph of rank $k$}]\label{defng}
Consider the abelian semigroup $\mathbb{N}^{k}$ (we'll assume that $0\in \mathbb{N}$). Viewing $\mathbb{N}^{k}$ as a category with one object, with arrows the elements in $\mathbb{N}^{k}$, and with composition of arrows $\mathbf a,\mathbf b\in \mathbb{N}^{k}$ is the sum $\mathbf a+\mathbf b\in \mathbb{N}^{k}.$ 
A \textit{graph of rank $k$} (or $k$-graph for short) is a countable category $\F$ together with a functor $\mathfrak{r}:\F\to \mathbb{N}^{k}$, called the \textit{rank functor}, such that if $\mathfrak{r}(f)=\mathbf a+\mathbf b\in \mathbb{N}^{k},$ for $\mathbf a,\mathbf b\in\mathbb{N}^{k}$ then there are \textit{unique} arrows $g$ and $h$ such that $\mathfrak{r}(g)=\mathbf a$, $\mathfrak{r}(h)=\mathbf b$ and $f=g\circ h.$
\end{defn}
It's easy to see from the definitions that the pre-image $\mathfrak{r}^{-1}(0)$ is the set of identity arrows of $\F.$ Moreover, the category $\F$ has the property that if $f=g\circ h=g'\circ h$ then $g=g'$, and similary if $f=g\circ h=g\circ h'$, then $h=h'$, i.e. every arrow in $\F$ is both a monomorphism and an epimorphism.'


Following~\cite{hrank}, we introduce some notations and terminology about higher rank graphs:
\\

Let $\F$ be a graph of rank $k.$
\begin{itemize}
\item For $\mathbf{n}\in \mathbb{N}^{k},$ let $\F^\mathbf{n}:=\rk^{-1}(\mathbf{n}).$ It's easy to see that we can identify $\F^{0}$ with $\Fo.$
\item For $f\in \Fa$, let $r(f):=\mathrm{cod}(f)\in \Fo$ and $s(f):=\mathrm{dom}(f)\in \Fo$ (range and source maps).
\item For $E\subseteq \Fa$ and $f\in \Fa$, let $f E:=\{f\circ g\;|\;g\in E,\, r(g)=s(f)\}$ and similarly $Ef:=\{g\circ f\;|\; s(g)=r(f)\, g\in E\}.$ Moreover, for $v\in \Fo$ we write $vE$ for the set $\id_{v}\F=\{g\in E\;|\; r(g)=v\},$ i.e. the set of elements in $E$ with range $v.$ Similarly we write $Ev$ for the set $\{g\in E\;|\; s(g)=v\}$ 
\item If $f\in \Fa$ and $\mathbf{n}\leq \mathbf{m}\leq \mathbf{j}=\rk(f)$ in the partial ordering of $\mathbb{N}^{k},$ then we write $f(0,\mathbf n),$ $f(\mathbf n,\mathbf m)$ and $f(\mathbf m,\mathbf j)$ for the unique paths of degree $\mathbf n$, $\mathbf m-\mathbf n$ and $\mathbf j-\mathbf m$ respectively, such that $f=f(0,\mathbf n)\circ f(\mathbf n,\mathbf m)\circ f(\mathbf m,\mathbf j).$ 
\item We call a $k$-graph $\F$ \textit{row-finite} if $|v\F^{\mathbf n}|<\infty$ for all $\mathbf n\in \mathbb{N}^{k}$ and $v\in \Fo.$
\item For $\mathbf{n}=(n_{1},\dots,n_{k})\in\mathbb{N}^{k},$ we write $\F^{\leq \mathbf{n}}$ for the set
$$
\F^{\leq \mathbf{n}}=\{\text{$f\in \F \;|\; \rk(f)\leq \mathbf{n}$ and $\rk(f)_{i}< n_{i}\Longrightarrow s(f)\F^{e_{i}}=\emptyset$}\},
$$
where the subscript $i$ denote the $i$'th coordinate in $\mathbb{N}^{k}$ and $e_{i}\in \mathbb{N}^{k}$ is the element with a $1$ in the $i$'th coordinate and $0$ elsewhere.
\item We say that $\F$ is \textit{locally convex} if whenever $j\neq i$, $f\in\F^{e_{j}}$, $g\in \F^{e_{i}}$ and $r(f)=r(g),$ we have $s(f)\F^{e_{i}}\neq \emptyset$ and $s(g)\F^{e_{j}}\neq \emptyset.$
\end{itemize}
\begin{rem}
The higher rank graphs considered in this text will have the property that $v\F^{e_{i}}\neq\emptyset$ and $\F^{e_{i}}v\neq\emptyset,$ for every $v\in \Fo$ and $i=1,\dots,k.$ Thus, they will be locally convex, and we'll have $\F^{\leq \mathbf{n}}=\F^{\mathbf{n}}.$
\end{rem}

We give the definition of a higher-rank graph $C^{*}$-algebra:
\begin{defn}
Let $\F$ be a locally convex, row-finite graph of rank $k.$ The $C^{*}$-algebra $C^{*}(\F)$ is the universal enveloping $C^{*}$-algebra generated by orthogonal projections $\{P_{\sigma}\;|\; \sigma \in \Fo\}$ and partial isometries $\{S_{f}\;|\; f\in \Fa\}$ subject to the Cuntz-Krieger relations:
\begin{enumerate}[(CK1)]
\item The projections $\{P_{v}\;|\; v \in \Fo\}$ are all mutually orthogonal,
\item $S_{f}S_{g}=S_{f\circ g}$ whenever $s(f)=r(g),$
\item $S^{*}_{f}S_{f}=P_{s(f)}$ for all $f\in \Fa$, and
\item $P_{v}=\sum_{f\in v\F^{\leq \mathbf n}}S_{f}S_{f}^{*}$ for all $v\in \Fo$ and $\mathbf n\in\mathbb{N}^{k}.$
\end{enumerate}
\end{defn}

For a $k$-graph $\F,$ we have a natural homomorphism $\alpha:\mathbb{T}^{k}\to Aut(C^{*}(\F))$ determined on generators as 
\begin{equation}\label{alphaT}
\alpha_{z}(S_{f})=t^{\mathfrak{r}(f)}S_{f},
\end{equation}
where $t^{\mathfrak{r}(f)}=(t_{1}^{\mathfrak{r}(f)_{1}},\dots t_{k}^{\mathfrak{r}(f)_{k}})$ (with $t^{0}=1$). Clearly, by (CK2) and the functorial property of $\mathfrak{r}$ for every $t\in \mathbb{T}^{k},$ this map extends to an automorphism of $C^{*}(\F).$
\begin{defn}
The set of homomorphisms $\{\alpha_{t}\,|\,t\in \mathbb{T}^{k}\}$ is called the {\it gauge action} on $C^{*}(\F)$
\end{defn}

A useful result in the representation theory of higher rank graph $C^{*}$-algebras is the 'Gauge Invariance Uniqueness Theorem' due to A. Kumjian and D. Pask
\begin{thm}[Theorem 3.4 in~\cite{hrank}]\label{gauge}
Let $\F$ be a $k$-graph, let $B$ be a $C^{*}$-algebra, and $\pi:C^{*}(\F)\to B$ be a homomorphism. If we have an action $\beta:\mathbb{T}^{k}\to Aut(B)$ such that $\pi\circ\alpha_{t}=\beta_{t}\circ\pi $ for all $t\in \mathbb{T}^{k}.$ Then $\pi$ is faithful if and only if $\pi(P_{\sigma})\neq 0$ for all $\sigma\in\F^{0}.$
\end{thm}

\section{$SU_{q}(3)$ as a Rank-2 Graph $C^{*}$-algebra}

The proof of the main result is divided into two steps:
\begin{enumerate}[$(i)$]
\item Give a proper definition of the $C^{*}$-algebra $$SU_{0}(3):=\lim_{q\to 0} \, SU_{q}(3),$$ and finding an appropriate 2-graph $\G$ such that $SU_{0}(3)\cong C^{*}(\G)$. 
\item Show that $SU_{0}(3)\cong SU_{0}(3)$ for all $0<q<1.$
\end{enumerate}
\begin{rem}
Step (i) is similar to what was done more generally in~\cite{roma}. The obstacle to generalizing the main result here is prove (ii).
\end{rem}
Here, we prove (i) by explicitly finding a set of partial isometries that generates $C_{0}(SU_{3})$, and that can be seen as the edges of a 2-graph. We then show that $C_{0}(SU_{3})$ is the universal $C^{*}$-representation for this graph. We then prove (ii) by essentially extending the proof from~\cite{giselsson} to the $q=0$ case.
\subsection{The Hopf $*$-algebra $\mathbb{C}[SU(n)]_{q}$}
Recall the definitions of the $*$-Hopf algebras $\mathbb{C}[SU(n)]_{q},$ for $q\in (0,1).$ It's defined by generators $\{t_{ij}\,|\, i,j=1,\dots,n\},$ along with a unit $I,$ subject to the relations
$$
\begin{array}{ccc}
t_{ij}t_{kl}-qt_{kl}t_{ij}=0, & \text{for $i=k$ and $
j<l$, or $i<k$ and  $j=l$,}\\
t_{i j}t_{k l}-t_{k l}t_{i j}=0,& \text{if $i<k$ and $j>l$,}\\
 t_{i j}t_{k l}-t_{kl}t_{ij}-(q-q^{-1})t_{il}
t_{kj}=0,& \text{for $i<k$ and $j<l$,}\\
\end{array}
$$
$$
\det \nolimits_q \mathbf{t}=I.
$$
Here $\det_q \mathbf{t}=\sum_{\sigma\in S_{n}}(-q)^{\ell(\sigma)}t_{1,\sigma(1)}\cdots t_{n,\sigma(n)}$ (where $S_{n}$ is the permutations of $n$ elements and $\ell$ the length of the permutation) is the  q-determinant of the matrix $\mathbf{t}=(t_{i,j})_{i,j=1}^{n}$.
 The comultiplication $\Delta$, the counit $\varepsilon$, the antipode $S$ and the involution $*$ are defined as follows
\begin{equation*}
\Delta(t_{ij})=\sum_k t_{ik}\otimes t_{kj},\quad \varepsilon(t_{ij})=\delta_{ij},\quad S(t_{ij})=(-q)^{i-j}\det\nolimits_q\mathbf{t}_{ji},
\end{equation*}
 and
\begin{equation*}\label{star5}
t_{ij}^*=(-q)^{j-i}\det \nolimits_q\mathbf{t}_{ij},
\end{equation*}
where ${\mathbf t}_{ij}$ is the matrix derived from ${\mathbf t}$ by discarding its $i$-th row and $j$-th column. 
\subsection{Representation Theory of $\mathbb{C}[SU(3)]_{q}$}
We give a quick rundown of the representation theory of $\mathbb{C}[SU(3)]_{q}$. The proofs of the statements here can be found in~\cite{KorS}.
\\

For two $\rho_{i}:\mathbb{C}[SU(3)]_{q}\to \mathcal{B}(\Hh_{i})$, $i=1,2$, we use the box-times notation to denote their tensor product 
$$
\rho_{1}\boxtimes \rho_{2}:=(\rho_{1}\otimes \rho_{2})\circ \Delta:\mathbb{C}[SU(3)]_{q}\to \mathcal{B}(\Hh_{1}\otimes \Hh_{2}).
$$

We give the usual construction of $\mathbb{C}[SU(3)]_{q}$-representations: Let $C_q,S,D_q:\ell^2(\mathbb{N})\to\ell^2(\mathbb{N})$ be the operators defined on the natural orthonormal basis $\{x_{j}\}_{j\in \mathbb{N}}$ as follows:
\begin{equation}\label{SC}
\begin{array}{cccc}
Sx_n=x_{n+1},& C_q x_n=\sqrt{1-q^{2n}}x_n,& D_q x_n=q^n x_n.
\end{array}
\end{equation}

By~\cite{wor}, the map
\begin{equation}\label{SC1}
\begin{array}{cccc}
\Pi_{q}(t_{11})=S^*C_q,&  \Pi_{q}(t_{12})=q D_q,&  \Pi_{q}(t_{21})=-D_q,&  \Pi_{q}(t_{22})=C_q S
\end{array}
\end{equation}
extends to a $*$-representation of $\mathbb C[SU(2)]_q.$ Let $C^{*}(S)\subseteq \mathcal{B}(\ell^{2}(\mathbb{N}))$ be the $C^{*}$-algebra generated by $S$ (so that $C^{*}(S)$ is the Toeplitz algebra). From the expressions~\eqref{SC} for $C_{q}$ and $D_q$ it is easy to see that both $C_{q},D_{q}\in C^{*}(S)$ and hence that $$\Pi_{q}(\mathbb C[SU(2)]_q)\subseteq C^{*}(S).$$
For $i=1,2,$ we have the two homomorphims $\mathbb{C}[SU(3)]_{q}\overset{\vartheta_{i}}{\to} \mathbb{C}[SU(2)]_{q}$ corresponding respectively to the two embeddings 
$$
\left[\begin{array}{ccc}
SU(2) & 0\\ 0 & 1
\end{array}\right],\,
\left[\begin{array}{ccc}
1 & 0\\ 0& SU(2)
\end{array}\right]\subseteq SU(3),
$$
and that is determined on the generators as
$$
\vartheta_{i}(t_{jk})=
\begin{cases}
t_{(j-i+1)(k-i+1)}, & \text{for $j,k\in\{1+i-1,2+i-1\}$}\\
\delta_{jk}I, & \text{otherwise.}
\end{cases}
$$

For $i=1,2,$ we define two representations
\begin{equation}\label{pii}
\pi_{i}=\Pi\circ \vartheta_{i}:\mathbb{C}[SU(3)]_{q}\to C^{*}(S).
\end{equation}

Moreover, we have two natural homomorphisms $\tau_{i}:\mathbb{C}[SU(3)]_{q}\to C(\mathbb{T})$ for $i=1,2$ determined by composition
$$
\tau_{i}:\mathbb{C}[SU(3)]_{q}\overset{\vartheta_{i}}{\to}C^{*}(S)\overset{\mathcal{K}\sim  }{\to} C^{*}(S)/\mathcal{K}\cong C(\mathbb{T}),
$$
using the standard isomorphism $C^{*}(S)/\mathcal{K}\cong C(\mathbb{T})$, where $\mathcal{K}\subseteq \mathcal{B}(\ell^{2}(\mathbb{N}))$ denotes the compact operators. To be specific, we consider the image of $S$ under the homomorphism $C^{*}(S)\to C(\mathbb{T})$ to be the coordinate function $z.$ Thus, as $D_{q}$ and $C_{q}-I$ are compact, we have (using the letter $I$ again for the identity in $C(\mathbb{T})$) 
$$
\begin{array}{cccc}
\tau_{1}(t_{11})=z, & \tau_{1}(t_{22})=\bar{z}, & \tau_{1}(t_{33})=I,\\
\tau_{2}(t_{11})=I, & \tau_{2}(t_{22})=z, & \tau_{2}(t_{33})=\bar{z},
\end{array}
$$
and $\tau_{i}(t_{kj})=0$ for all other indices.
\\

 The homomorphism $\tau_{1}\boxtimes \tau_{2}:\mathbb{C}[SU(3)]_{q}\to C(\mathbb{T})\otimes C(\mathbb{T})$ corresponds to the maximal torus restriction $ \mathbb{T}^{2}\subseteq SU(3)$. We denote the universal enveloping $C^{*}$-algebra of $\mathbb{C}[SU(3)]_{q}$ by $SU_{q}(3).$ From the representation theory, we get that $SU_{q}(3)$ is isomorphic to the closure of $\mathbb{C}[SU(3)]_{q}$ under the faithful representation
\begin{equation}\label{uni}
\Xi_{q}:= (\pi_{1}\boxtimes \pi_{2}\boxtimes \pi_{1})\boxtimes (\tau_{1}\boxtimes \tau_{2}):\mathbb{C}[SU(3)]_{q}\to C^{*}(S)^{\otimes 3}\otimes C(\mathbb{T})\otimes C(\mathbb{T}).
\end{equation}
In this paper, we will usually identify $\mathbb{C}[SU(3)]_{q}$ with $\im\, \Xi_{q}$ and the $C^{*}$-algebra $SU_{q}(3)$ with its closure. This is to be able to simultaneously consider all these algebras as sub-algebras of $C^{*}(S)^{\otimes 3}\otimes C(\mathbb{T})\otimes C(\mathbb{T}).$
\subsection{The Algebras $\lim_{q\to 0}\mathbb{C}[SU(3)]_{q}$ and $\lim_{q\to 0}SU_{q}(3)$}
To make sense of the $C^{*}$-algebras $\mathbb{C}[SU(3)]_{0}$ and $SU_{0}(3)$ in the limit $q\to 0,$ we want to take an appropriate set of generators of $\mathbb{C}[SU(3)]_{q}$ where the limits $q\to 0$ under the representation~\eqref{uni} actually exists. We then define $\mathbb{C}[SU(3)]_{0}$ as $*$-algebra that these limits generates, and $SU_{0}(3)$ its $C^{*}$-closure. 
\\

Let $U_{q}(\mathfrak{su}(3))$ be the Jimbo-Drinfeld $q$-deformation of the universal algebra of $\mathfrak{su}(3),$ and let $V_{+}$ denote the set of dominant weights. For $\lambda\in V_{+},$ we let $M_{\lambda}^{q}$ denote the $U_{q}(\mathfrak{su}(3))$-module with heighest weight corresponding to $\lambda.$ We let $\omega_{i}\in V_{+},$ for $i=1,2,$ denote the fundamental weights, ordered such that
$$
\begin{array}{ccc}
\omega_{1}=(1,-\frac{1}{2},-\frac{1}{2}), & \omega_{2}=(\frac{1}{2},\frac{1}{2},-1).
\end{array}
$$
As $\dim\, M_{\omega_{i}}^{q}=3,$ for $i=1,2,$ we let $\{\xi_{\omega_{i}}^{1},\xi_{\omega_{i}}^{1},\xi_{\omega_{i}}^{1}\}\subseteq M_{\omega_{i}}^{q}$ be an orthonormal basis, such that $\xi_{\omega_{i}}^{1}$ generates the heighest weight space, $\xi_{\omega_{i}}^{3}$ generates the lowest.
\\
Using the identification of $\mathbb{C}[SU(3)]_{q}$ with the reduced dual of $U_{q}(\mathfrak{su}(3))$, we consider the elements in $\mathbb{C}[SU(3)]_{q}$ defined as 
\begin{equation}\label{asd}
\begin{array}{cccc}
C_{j}^{\omega_{i}}(\cdot)=\langle \cdot\; \xi_{\omega_{i}}^{j},\xi_{\omega_{i}}^{1}\rangle , &i=1,2 ,& j=1,2,3.
\end{array}
\end{equation}
By (~\cite{KorS}, Theorem 2.2.1), the elements~\eqref{asd} generates $\mathbb{C}[SU(3)]_{q}$ as a $*$-algebra. Moreover, by perhaps multiplying these element with an appropriate value in $\mathbb{T},$ it's not hard to see that one can express these elements in the usual set of generators $t_{ij}$ as
\begin{equation}\label{generators}
\begin{array}{ccc}
C_{j}^{\omega_{1}}=t_{j1}, & C_{j}^{\omega_{2}}=q^{(j-3)}t_{(4-j)3}^{*}.
\end{array}
\end{equation} 
We omit the proof of this statement, as we will do explicit calculations involving the right-hand sides of~\eqref{generators}. It's a simple calculation to show that these generates $\mathbb{C}[SU(3)]_{q}$ as a $*$-algebra. The formula~\eqref{asd} is there to give a more structural understanding of whats going on. We will thus assume that~\eqref{generators} holds, and consider it as the definition of $C_{j}^{\omega_{i}}.$
\\

Let $P\in \mathcal{B}(\ell^{2}(\mathbb{N}))$ denote the orthogonal projection onto the subspace generated by $e_{0}$, and we let $Q=I-P.$ Moreover, let $z\in C(\mathbb{T})$ be the coordinate function. Taking $q\to 0,$ we have limits in norm (we spare the reader the standard -- but tedious -- calculations)
\begin{align}\label{a}
\Xi_{q}(C_{1}^{\omega_{2}})^{*}&\leadsto  I\otimes S\otimes I\otimes z\otimes I&=\hat A\\
\Xi_{q}(C_{2}^{\omega_{2}})^{*}&\leadsto   I\otimes P\otimes S\otimes z\otimes I&=\hat B\\
\Xi_{q}(C_{3}^{\omega_{2}})^{*}&\leadsto  I\otimes P\otimes P\otimes z\otimes I&=\hat C
\end{align}

\begin{align}
\Xi_{q}(C_{1}^{\omega_{1}})^{*}&\leadsto  S\otimes I\otimes S\otimes I\otimes z&= \hat{X} \\
\Xi_{q}(C_{2}^{\omega_{1}})^{*}&\leadsto  S\otimes I\otimes P\otimes I\otimes z+ P\otimes S\otimes S^{*}\otimes I\otimes z&=\hat Y\\
\Xi_{q}(C_{3}^{\omega_{1}})^{*}&\leadsto  P\otimes P\otimes I\otimes I\otimes z&=\hat Z \label{z}
\end{align}

These are all partial isometries, as 
\begin{gather*}
\hat A^{*}\hat A= I\otimes I\otimes I\otimes I\otimes I, 
\qquad
\hat A\hat A^{*}= I\otimes Q\otimes I\otimes I\otimes I,\\
 \hat B^{*}\hat B= I\otimes P\otimes I\otimes I\otimes I,
\qquad
\hat B\hat B^{*}=I\otimes P\otimes Q\otimes I\otimes I,\\
\hat C^{*}\hat C=\hat C\hat C^{*}= I\otimes P\otimes P\otimes I\otimes I,
\end{gather*}
\begin{gather*}
\hat X^{*}\hat X=  I \otimes I\otimes I\otimes I\otimes I, 
\qquad
\hat X\hat X^{*}= Q\otimes I\otimes Q\otimes I\otimes I,\\
\hat Y^{*}\hat Y= I\otimes I\otimes P\otimes I\otimes I+  P\otimes I\otimes Q\otimes I\otimes I,\\
\hat Y\hat Y^{*}= Q\otimes I\otimes P\otimes I\otimes I+ P\otimes Q\otimes I\otimes I\otimes  I,\\
\hat Z^{*}\hat Z=\hat Z\hat Z^{*}= P \otimes P \otimes I\otimes I\otimes I.
\end{gather*}
Recalling that a partial isometry $v$ is called \textit{quasinormal} if $vv^{*}\leq v^{*}v,$ we see from the above that $A,B,C,X,Y,Z$ are indeed quasinormal partial isometries.
\\
Moreover, we see that $$\hat A\hat A^{*}+\hat B\hat B^{*}+\hat C\hat C^{*}=I=\hat X\hat X^{*}+\hat Y\hat Y^{*}+\hat Z\hat Z^{*}.$$ 
\begin{defn}
We denote the $*$-algebra generated by the limits~\eqref{a}- \eqref{z} by $\mathbb{C}[SU(3)]_{0}$, and denote its norm-closure by $SU_{0}(3)$.
\end{defn}

The right-action $\beta_{t}$, for $t=(t_{1},t_{2})\in \mathbb{T}^{2}$ of the maximal torus on $C_{j}^{\omega_{i}}$, is induced by $\beta_{t}(C_{j}^{\omega_{i}})=t_{i}C_{j}^{\omega_{i}}$. Taking any non-commutative polynomials $F$ in the variables $C_{j}^{\omega_{i}},(C_{j}^{\omega_{i}})^{*}$, we get from the limit
$$
\lim_{q\to 0}|| \Xi_{q}(F)||=\lim_{q\to 0}|| \Xi_{q}(\beta_{t}(F))||,
$$
 (since $F$ was arbitrary) that $\beta_{t}$ extends to an isormorphism of $SU_{0}(3)$ induced by the actions 
\begin{equation}\label{betagen}
\begin{array}{ccc}
  \{\hat A,\hat B,\hat C\}\mapsto \{t_{1}\hat A,t_{1}\hat B,t_{1}\hat C\},&\{\hat X,\hat Y,\hat Z\}\mapsto \{t_{2}\hat X,t_{2}\hat Y,t_{2}\hat Z\},& \text{for $t=(t_{1},t_{2}).$}
\end{array}
\end{equation}
From~\eqref{betagen} it follows that $\beta_{ts}=\beta_{t}\circ \beta_{s}.$ Moreover, as the map $t\in \mathbb{T}^{2}\mapsto \beta_{t}(\hat{R})$ is norm-continous when $\hat{R}\in\{\hat{A},\hat{B},\hat{C},\hat{X},\hat{Y},\hat{Z}\}$, we get from an standard approximation argument that also $t\mapsto \beta_{t}(a)$ is norm-continuous for every $a\in SU_{0}(3).$ Thus we have an injective homomorphism $\beta_{t}:\mathbb{T}^{2}\to Aut(SU_{0}(3))$ that is point-norm continuous.
By using the natural identification 
$$\mathcal{T}^{\otimes 3}\otimes C(\mathbb{T})\otimes C(\mathbb{T})\cong C(\mathbb{T}^{2}:\mathcal{T}^{\otimes 3}),$$
(i.e. continuous functions $\mathbb{T}^{2}\to \mathcal{T}^{\otimes 3}$) we see that for $f(z_{1},z_{2})\in SU_{q}(3)\subseteq C(\mathbb{T}^{2}:\mathcal{T}^{\otimes 3}),$ and all $q\in[0,1)$, we have
\begin{equation}\label{beta}
\beta_{t}(f(z_{1},z_{2}))=f(t_{1}z_{1},t_{2}z_{2}).
\end{equation}
Thus $\beta$ is actually the restriction of the $\mathbb{T}^{2}$-action $(f(z_{1},z_{2}))\mapsto f(t_{1}z_{1},t_{2}z_{2})$ on the ambient $C^{*}$-algebra (identified with) $C(\mathbb{T}^{2}:\mathcal{T}^{\otimes 3})$. We can thus, with a slight abuse of notation, use the same symbol $\beta$ to denote all these group actions.
\subsection{The Colored Graph of $SU_{0}(3)$}
We make the projections that are supposed to be the nodes of our graph labeled by the set $\{AX,AY,BX,BZ,CY,CZ\}$

\begin{align}\label{proj1}
\hat P_{CZ}=& \hat{C}\hat{C}^{*}\hat Z\hat Z^{*}= P\otimes P\otimes P\otimes I\otimes I,\\ 
\hat P_{CY}=& \hat C\hat C^{*}\hat Y\hat Y^{*}= Q\otimes P\otimes P\otimes I\otimes I,\\
\hat P_{BZ}=&\hat  B\hat B^{*}\hat Z\hat Z^{*}= P\otimes P\otimes Q\otimes I\otimes I,\\
\hat P_{BX}=&\hat B\hat B^{*}\hat X\hat X^{*}= Q\otimes P\otimes Q\otimes I\otimes I, \\
\hat P_{AY}=&\hat A\hat A^{*}\hat Y\hat Y^{*}= Q\otimes Q\otimes P\otimes I\otimes I+ P\otimes Q\otimes I\otimes I\otimes  I,\\
\hat P_{AX}=&\hat A\hat A^{*}\hat X\hat X^{*}= Q\otimes Q\otimes Q\otimes I\otimes I. \label{proj6}
\end{align}

A quick calculation shows that these projections actually adds up to the identity.
\\

We will construct a directed $2$-colored graph $\mathcal{G}=(V,E,r,s)$ for which these projections will correspond to the vertices of. The vertice will be labelled by the corresponding index in the projection. Thus the  $V=\{AX, AY, BX,BZ, CY,CZ\}.$
The set of edges $E$ is defined and labelled by $\{ A, B ,C,X,Y,Z\}$ in the following way: There is directed edge $$v_{1}\overset{e}{\to} v_{2},$$ labelled by $e\in \{A,B,C,X,Y,Z\}$  if we have
\begin{equation}\label{ev}
\hat P_{v_{2}}\hat{e}\hat  P_{v_{1}}\neq 0. 
\end{equation}
 Moreover, we will color the edges either {\color{red} red} or {\color{blue} blue} (dashed) depending on if the label takes values in the sets $\{A,B,C\}$ or $\{X,Y,Z\}$ respectively. It's easily verifiable that resulting graph $\mathcal{G}$ looks as follows:
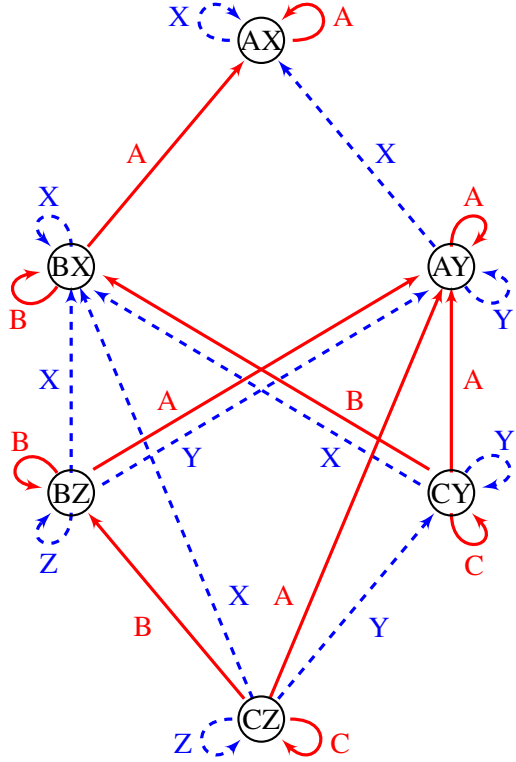
\begin{equation}\label{2graph}
\begin{tikzpicture}[>=latex',node distance = 1.5cm]'
\node at (0,0) (CZ) {CZ};
\node at (2.5,3) (CY) {CY};
\node at (-2.5,3) (BZ) {BZ};
\node at (-2.5,6) (BX) {BX};
\node at (2.5,6) (AY) {AY};
\node at (0,9) (AX) {AX};
\draw [dashed,->,very thick, blue] (CZ) to node [below,xshift= 0.3cm]{Y}  (CY);
\draw [->,very thick, red] (CZ) to  node [below,xshift= -0.3cm]{B}  (BZ);
\draw [dashed,->,very thick, blue] (BZ) to node[left] {X}  (BX);

\draw [->,very thick,red] (2.5-0.35+0.06,3+0.21+0.1) to node[near start,above,xshift=0.1cm] {B}  (-2.5+0.35+0.06,6-0.21+0.1);
\draw [dashed,->,very thick,blue] (2.5-0.35-0.06,3+0.21-0.1) to node[near start,below,xshift=-0.1cm] {X}  (-2.5+0.35-0.06,6-0.21-0.1);

\draw [->,very thick,red] (-2.21,3.31) to node[near start,above,xshift=-0.1cm] {A}  (2.09,5.89);
\draw [dashed,->,very thick,blue] (-2.09,3.11) to node[near start,below,xshift=0.1cm] {Y}  (2.21,5.69);

\draw [->,very thick,red] (CY) to  node[right] {A} (AY);
\draw [dashed,->,very thick, blue] (AY) to node[right,xshift=0.1cm]{X}  (AX);
\draw [->,very thick, red] (BX) to node[left,xshift=-0.1cm]{A} (AX);

 \draw[dashed,->,very thick,blue] (CZ) to[out=180,in=225,looseness=6] node[left] {Z} (CZ);
 \draw[->,very thick, red] (CZ) to[in=-45,out=0,looseness=6] node[right] {C} (CZ);

 \draw[dashed,->,very thick,blue] (BX) to[out=90,in=90+45,looseness=6] node[above] {X} (BX);
 \draw[->,very thick, red] (BX) to[in=90+100,out=90+45+100,looseness=6] node[below] {B} (BX);

 \draw[->,very thick,red] (BZ) to[out=125,in=125+45,looseness=6] node[above] {B} (BZ);
 \draw[dashed,->,very thick, blue] (BZ) to[in=125+100,out=125+45+100,looseness=6] node[below] {Z} (BZ);

 \draw[->,very thick,red] (CY) to[out=90+180,in=90+45+180,looseness=6] node[below] {C} (CY);
 \draw[dashed,->,very thick, blue] (CY) to[in=90+100+180,out=90+45+100+180,looseness=6] node[above] {Y} (CY);

 \draw[dashed,->,very thick,blue] (AY) to[out=125+180,in=125+45+180,looseness=6] node[below] {Y}  (AY);
 \draw[->,very thick, red] (AY) to[in=125+100+180,out=125+45+100+180,looseness=6] node[above] {A} (AY);

 \draw[dashed,->,very thick,blue] (AX) to[in=-225,out=-180,looseness=6] node[left] {X} (AX);
 \draw[->,very thick, red] (AX) to[out=0,in=45,looseness=6] node[right] {A} (AX);

\draw [dashed,->,very thick, blue] (CZ) to node[near start,right,xshift=0.1cm]{X}  (BX);
\draw [->,very thick, red] (CZ) to node[near start,left,xshift=-0.1cm]{A} (AY);

\foreach \n in {AX,AY,BX,BZ,CY,CZ}
  \draw[black, thick] (\n) circle (0.3 cm);
\end{tikzpicture}
\end{equation}
This will be the blueprint for the $2$-graph that we construct below. There will be further relations in the compositions of the operators $\hat A,\dots, \hat Z$ not captured by the graph~\eqref{2graph}. As an example, it's not hard to see that $\hat Y\hat B=\hat A\hat Z.$ The same holds for all admissible compositions of arrows labelled such in the $2$-graph category. 
\begin{rem}
This text is quite heavy on the calculations, and all throughout it we will use~\eqref{2graph} as a calculatory tool.
\end{rem}
\subsection{Construction of the 2-graph $\mathcal{Su}(3)$}
We will use the methods from Section 6 in~\cite{hrank} to construct the $2$-graph. The basis for their contruction is a directed colored graph, such that the transition matrices commutes for different colors, together with further commutation rules whenever there is ambiguity in the arrow composition. 
\\

Let us consider the complex vector space $\mathbb{C}^{6}.$ We identify the canonical basis vectors $v_{1},\dots,v_{6}$ with the vertices of~\eqref{2graph} in the following way:
$$
\begin{array}{cccccc}
1=CZ, & 2=BZ, & 3=CY,& 4=BX,& 5=AY, &6=AX.
\end{array}
$$
We define the transition matrices $M_{R}$ and $M_{B}$ ($R$ for red and $B$ for blue) such that $M_{R}$ has a $1$ at the $(i,j)$ index if in~\eqref{2graph}, we have a red arrow $i\to j$ in the corresponding vertices. Similar for $M_{B}$. With this definition, we have
\begin{equation}\label{matrb}
\begin{array}{cc}
M_{R}=\left[\begin{array}{cccccc}1&0&0&0&0&0\\1&1&0&0&0&0\\0&0&1&0&0&0\\1&1&1&1&0&0\\0&0&1&0&1&0\\0&0&0&0&1&1\end{array}\right],&
M_{B}=\left[\begin{array}{cccccc}1&0&0&0&0&0\\0&1&0&0&0&0\\1&0&1&0&0&0\\0&1&0&1&0&0\\1&1&1&0&1&0\\0&0&0&1&0&1\end{array}\right].
\end{array}
\end{equation}
A calculation gives that
\begin{equation}\label{comrb}
M_{R}M_{B}=M_{B}M_{R}=\left[\begin{array}{cccccc}1&0&0&0&0&0\\1&1&0&0&0&0\\1&0&1&0&0&0\\2&2&1&1&0&0\\2&1&2&0&1&0\\1&1&1&1&1&1\end{array}\right]
\end{equation}
so that these matrices fulfill the requirements in [Section 6, in~\cite{hrank}]. We see that in~\eqref{comrb}, we have $4$ indices with $2$'s in them, and we need to resolve this ambiguity.
\\
Let us define
$$
A_{RB}^{ij}=\{(a,x)\,|\,\text{$a$ a red arrow, $x$ blue, such that $r(a)=i,$ $s(a)=r(x)$ and $s(x)=j$}\}.
$$
and we define $A_{BR}^{ij}$ similarly, but with the colors switched. From~\eqref{comrb}, we deduce that the sets $A_{RB}^{ij}$ and $A_{BR}^{ij}$ contains the same number of elements for all indicies $i,j$. For each $i,j$, we will construct a bijection
$$
\phi_{ij}:A_{BR}^{ij}\to A_{RB}^{ij}.
$$
It is clear from~\eqref{comrb} that we only need to specify $\phi_{ij}$ for $i,j=(1,4),(1,5),(2,4),(3,5),$ i.e. for
\begin{equation}\label{phimap1}
\begin{array}{cccc}
(CZ)\to (AY), & (BZ)\to (AY),
\end{array}
\end{equation}
\begin{equation}\label{phimap2}
\begin{array}{ccc}
(CZ) \to (BX), & (CY)\to (BX).
\end{array}
\end{equation}
We get these  from the relations satisfied by the operators~\eqref{a}-\eqref{z}  as (for the arrows with the labels)
\begin{equation}\label{phimap}
\begin{array}{cccc}
Z\circ A=B\circ Y, & Y\circ A=A\circ Y, &\text{for $\phi_{15},\phi_{35}$ (the arrows from~\eqref{phimap1})},\\
Y\circ B=C\circ X, & X\circ B=B\circ X, & \text{for $\phi_{14},\phi_{24}$ (the arrows from~\eqref{phimap2}).}
\end{array}
\end{equation}
From section 6 in~\cite{hrank}, we now get
\begin{prop}\label{su3graph}
This input-data uniquely determines a $2$-graph $\G$.
\end{prop}
Thus $\G$ is the $2$-graph with verties the nodes in~\eqref{2graph}, the blue and red arrows the sets $\mathfrak{r}^{-1}(1,0)$ and $\mathfrak{r}^{-1}(0,1)$ respectively, and where we have the additional relations between the red and blue arrows determined by~\eqref{phimap}. 
\subsection{Proof that $C^{*}(\G)\cong SU_{0}(3)$}
We define operators $A,B,C,X,Y,Z\in C^{*}(\G)$ by letting these be the sum of all the partial isometries in $C^{*}(\G)$ corresponding to the edges in~\eqref{2graph} labelled respectively. We denote the orthogonal projections in $C^{*}(\G)$ coming from the vertices by $P_{v},$ with $v\in \{AX,AY,BX,BZ,CY,CZ\}.$
\begin{lem}\label{gens}
The elements $A,..,Z$ is generating $C^{*}(\G).$
\end{lem}
\begin{proof}
Note that all the red or blue arrows with range a particular vertex has the same labels. By inspection of the graph~\eqref{2graph}, and (CK4), we have
\begin{equation}\label{projs}
\begin{array}{ccc}
AA^{*}=P_{AX}+P_{AY}\\
BB^{*}=P_{BX}+P_{BZ}\\
CC^{*}=P_{CY}+P_{CZ}\\
\\
XX^{*}=P_{AX}+P_{BX}\\
YY^{*}=P_{AY}+P_{CY}\\
ZZ^{*}=P_{BZ}+P_{CZ}.
\end{array}
\end{equation}
From the list~\eqref{projs}, we see that with $m\in\{A,B,C\}$ and $n\in \{X,Y,Z\}$, we have $mm^{*}nn^{*}=P_{mn}.$ It follows that we can recover all the edges in~\eqref{2graph}.
\end{proof}
Unsurprisingly, this is the same formula for the projection as~\eqref{proj1}-~\eqref{proj6}. We also see from~\eqref{projs} that
$$
AA^{*}+BB^{*}+CC^{*}=I=XX^{*}+YY^{*}+ZZ^{*}.
$$
\begin{lem}\label{relations}
$A,B,C,X,Y,Z$ are all quasinormal partial isometries. In particular, $A$ and $X$ are isometries. Moreover, the following relations hold
\begin{equation}\label{rel}
\begin{array}{ccc}
BA=B^{*}A=0, & YA=AY, & CA=C^{*}A=0\\
ZB=BZ, &YB=AZ,& ZA=Z^{*}A=0, \\
 CB=C^{*}B=0, & AX=XA, & A X^{*}=X^{*}A.\\
\\
YX=Y^{*}X=0, & BX=XB, & ZX=Z^{*}X=0\\
CY=YC, &BY=XC,& CX=C^{*}X=0, \\
 &ZY=Z^{*}Y=0,&
\end{array}
\end{equation}
\end{lem}
\begin{proof}
Notice that by the symmetry of the graph~\eqref{2graph} and the relations~\eqref{phimap}, we can switch $A,B,C$ with $X,Y,Z$. By this argument, the relations below the gap in~\eqref{rel} follows from the ones above. 
\\

Let us first show the claims of quasinormality and isometry. We know from Lemma~\ref{gens} that $A,\dots,Z$ are partial isometries. Moreover, as every vertex is a source vertex for an $A$ or $X$ arrow, it follows that $A$ and $X$ are actually isometries. Quasinormality (i.e. $vv^{*}\leq v^{*}v$) is easy to see from (CK3) and (CK4), as, for example in the case of $Y$, the set of range vertices for $Y$-labelled arrows is a subset of of the source vertices. Similarly for the others.
\\

Assume for this paragraph that $A,..,Z$ again denotes the arrow-labels in~\eqref{2graph}. The definition of $\G$, in particular~\eqref{phimap}, shows that $YA=AY$, $ZB=BZ$, $YB=AZ$ and $AX=XA$ for all composable arrows with these labels. Some diagram-inspection shows that the same is true for the corresponding operators.
\\

We deduce that $BA=B^{*}A=CA=C^{*}A=ZA=Z^{*}A=0$ from the fact that in~\eqref{2graph}, no arrow with label $A$ has its range on any vertex that is the domain or range of an arrow labeled either $B,$ $C,$ or $Z.$ We have $CB=C^{*}B=0$ for the same reason.
\\

Lets prove the last relation $A X^{*}=X^{*}A$. Some diagram chasing of~\eqref{2graph} gives that the projection onto $\ker\, A^{*}$ is given by $P=P_{BX}+P_{BZ}+P_{CY}+P_{CZ}.$ By~\eqref{2graph}, this subspace is invariant under $X.$ It follows that $$PXP=XP\Rightarrow (I-P)X(I-P)=(I-P)X.$$ As $AA^{*}=I-P,$ we get from $AX=XA$ that 
$$XA^{*}=A^{*}(AX)A^{*}=A^{*}(XA)A^{*}=A^{*}X(I-P).$$ Then
$$
A^{*}X(I-P)=A^{*}AA^{*}X(I-P)=A^{*}(I-P)X(I-P)=A^{*}(I-P)X=A^{*}X.
$$
\end{proof}
Assume that we have a Hilbert space $\mathrm{H}$ and a faithful representation $$\rho:C^*(\G)\to \mathcal{B}(\mathrm{H}).$$
Let us denote $\mathcal E=\rho(P_{CZ})\mathrm H.$ To make things appear more concise, we will supress $\rho$ and write $x$ instead of $\rho(x)$ for $x\in C^{*}(\G).$ This is safe as $\rho$ was assumed to be faithful.
\\

By the conditions (CK3)-(CK4), we have that $\mathcal{E}$ is invariant under the operators $C$ and $Z,$ and moreover, the restrictions to this subspace are unitary.
Thus, let us denote $\tilde{C}:=C|_{\mathcal{E}}$ and $\tilde{Z}:=Z|_{\mathcal{E}}.$
Clearly $\tilde{C},\tilde{Z}\in \mathcal{B}(\mathcal{E})$ are commuting unitary operators. Moreover as the representation is assumed to be faithful, it follows from the gauge action that we have an isomorphism $C^{*}(\tilde{C},\tilde{Z})\cong C(\mathbb{T})\otimes C(\mathbb{T}),$ and moreover that this isomorphism can be taken so that $\tilde{C}\mapsto z\otimes I$ and $\tilde{Z}\mapsto I\otimes z.$
\\

Consider the closed subspaces
\begin{equation}\label{aby}
\begin{array}{ccc}
\mathcal{E}(\mathbf{k})=A^{k}B^{m}Y^{j}\tilde{C}^{-(k+m)}\tilde{Z}^{-j}\mathcal{E}, & \text{for $\mathbf{k}= (j,k,m)\in \mathbb{N}^{3}.$}
\end{array}
\end{equation}
Note that the negative exponents are not a problem as $\tilde{C}$ and $\tilde{Z}$ are unitary.

\begin{lem}
We have $\mathcal{E}(\mathbf{k})\bot \mathcal{E}(\mathbf{n})$ for $\mathbf{k}\neq \mathbf{n}$, and a unitary isometry
\begin{equation}\label{uni}
W:\oplus_{\mathbf{k}\in \mathbb{N}^{3}}  \mathcal{E}(\mathbf{k})\to  \ell^{2}(\mathbb{N})^{\otimes 3}\otimes \mathcal{E},
\end{equation}
given by the formula
\begin{equation}\label{map}
\begin{array}{cccc}
A^{k}B^{m}Y^{j}\tilde{C}^{-(k+m)}\tilde{Z}^{-j}x\mapsto  e_{j}\otimes e_{k}\otimes e_{m}\otimes x, & \text{for $x\in \mathcal{E}.$}
\end{array}
\end{equation}
\end{lem}
\begin{proof}
As $\tilde{C}$ and $\tilde{Z}$ are unitary, we can safely ignore these factors in~\eqref{aby} when we prove the orthogonality and non-zero-ness of the subspaces $\mathcal{E}(\mathbf{k})$. 
An inspection of the graph gives that
\begin{equation}\label{isos}
\begin{array}{ccc}
A^{*}A=I,& AA^{*}=P_{AX}+P_{AY},\\
B^{*}B=P_{BX}+P_{BZ}+P_{CY}+P_{CZ} ,& BB^{*}=P_{BX}+P_{BZ},\\
Y^{*}Y=P_{AY}+P_{BZ}+P_{CY}+P_{CZ} ,& YY^{*}=P_{AY}+P_{CY}.
\end{array}
\end{equation}
If follows from this that $A,B,Y$ are isometries when restricted to the ranges of $A^{*}A,B^{*}B,Y^{*}Y,$ respectively (note that these subspaces also reduces their respective operator).  Moreover, the Wold–von Neumann decomposition gives that with $\mathcal{E}_{A}=(I-P_{AX}-P_{AY})\mathrm{H},$ $\mathcal{E}_{B}=(P_{CY}+P_{CZ})\mathrm{H}$, and $\mathcal{E}_{Y}=(P_{AY}+P_{CY})\mathrm{H},$ that for $R\in \{A,B,Y\}$ and all $k\in \mathbb{N}$ we have a unitary operator
$$
\begin{array}{cccc}
R:R^{k}\mathcal{E}_{R}\to R^{k+1}\mathcal{E}_{R},
\end{array}
$$
and moreover
$$
\begin{array}{ccc}
R^{m}\mathcal{E}_{R}\bot R^{j}\mathcal{E}_{R}, & \text{for $m\neq j.$}
\end{array}
$$
An inspection of~\eqref{2graph} gives that we have
$$
\begin{array}{cccc}
 \mathcal{E}\subseteq \mathcal{E}_{Y},& Y^{k}\mathcal{E}\subseteq \mathcal{E}_{B}, & 
 B^{k}\mathcal{E}_{B}\subseteq \mathcal{E}_{A}, & \text{for all $k\in \mathbb{N}.$}
\end{array}
$$
From this we can then deduce that for all $\mathbf{k}=(j,k,m)\in \mathbb{N}^{3},$ the map 
$$
A^{k}B^{m}Y^{j}:\mathcal{E}\to \mathcal{E}(\mathbf{k})
$$
is unitary, and moreover that $\mathcal{E}(\mathbf{k})\bot \mathcal{E}(\mathbf{n})$ for $\mathbf{k}\neq \mathbf{n}$. In particular, it follows that $W$ determined by~\eqref{map} is unitary.
\end{proof}

\begin{prop}\label{form}
The subspace $\oplus_{\mathbf{k}\in \mathbb{N}^{3}}  \mathcal{E}(\mathbf{k})$ is reducing the images of $A,..,Z.$ Moreover, if $W$ is the isomorphism~\eqref{uni}, then 
\begin{equation}\label{formulas}
\begin{array}{ccc}
WAW^{*}= I\otimes S\otimes I\otimes \tilde{C},\\
WBW^{*}= I\otimes P\otimes S\otimes \tilde{C},\\
WCW^{*}= I\otimes P\otimes P\otimes \tilde{C}, \\
\\
WXW^{*}=S\otimes I\otimes S\otimes \tilde{Z},\\
WYW^{*}= S\otimes I\otimes P\otimes \tilde{Z}+ P\otimes S\otimes S^{*}\otimes \tilde{Z},\\
WZW^{*}= P\otimes P\otimes I\otimes \tilde{Z}.
\end{array}
\end{equation}
\end{prop}
\begin{proof}
We'll use the notation 
$$
\begin{array}{cccc}
[x](j,k,m):=A^{k}B^{m}Y^{j}\tilde{C}^{-(k+m)}\tilde{Z}^{-j}x, & \text{for $x\in\mathcal{E}$.}
\end{array}
$$
By using the list of relations~\eqref{rel} from Lemma~\eqref{relations}, it's simple to calculate the actions of $A,..,Z$ and $A^{*}-Z^{*}$ on this element as:
$$
A[x](j,k,m)=[\tilde{C}x](j,k+1,m)
$$
$$
A^{*}[x](j,k,m)=\begin{cases}0, & \text{if $i=0$}\\ [\tilde{C}^{*}x](j,k-1,m), & \text{otherwise}\end{cases},
$$
$$
B[x](j,k,m)=\begin{cases} 0, & \text{if $k>0$}\\ [\tilde{C}x](j,0,m+1), & \text{otherwise}\end{cases}
$$
$$
B^{*}[x](j,k,m)=\begin{cases} 0, & \text{if $k>0$ or $m=0$}\\ [\tilde{C}^{*}x](j,0,m-1), & \text{otherwise}\end{cases}
$$
$$
C[x](j,k,m)=\begin{cases} 0, & \text{if $k,m>0$}\\ [\tilde{C}x](j,0,0), &\text{otherwise}\end{cases}
$$
$$
C^{*}[x](j,k,m)=\begin{cases} 0, & \text{if $k,m>0$}\\ [\tilde{C}^{*}x](j,0,0), &\text{otherwise}\end{cases}
$$
$$
X[x](j,k,m)=XA^{k}B^{m}Y^{j}\tilde{C}^{-(k+m)}\tilde{Z}^{-j}x=A^{k}B^{m}XY^{j}\tilde{C}^{-(k+m)}\tilde{Z}^{-j}x=
$$
$$
=A^{k}B^{m}(XC)Y^{j}\tilde{C}^{-(k+m+1)}\tilde{Z}^{-j}x=A^{k}B^{m}(BY)Y^{j}\tilde{C}^{-(k+m+1)}\tilde{Z}^{-(j+1)}\tilde{Z}x=
$$
$$
=[\tilde{Z}x](j+1,k,m+1),
$$
\begin{equation}\label{calc}
X^{*}[x](j,k,m)=X^{*}A^{k}B^{m}Y^{j}\tilde{C}^{-(k+m)}\tilde{Z}^{-j}x=A^{k}X^{*}B^{m}Y^{j}\tilde{C}^{-(k+m)}\tilde{Z}^{-j}x.
\end{equation}
We break down the calculation of the right-hand side of~\eqref{calc} into three cases: 
\begin{enumerate}[(i)]
\item if $m=0$, then $A^{k}X^{*}Y^{j}\tilde{C}^{-k}\tilde{Z}^{-j}x=0,$
\item if $j=0$, then $A^{k}X^{*}B^{m}\tilde C^{-(k+m)}x=A^{k}(X^{*}Z)B^{m}\tilde{C}^{-(k+m)}\tilde{Z}^{-1}x=0,$
\item if $j,m>0,$ then 
$$
A^{k}X^{*}B^{m}Y^{j}\tilde{C}^{-(k+m)}\tilde{Z}^{-j}x=A^{k}X^{*}B^{m-1}(BY)Y^{j-1}\tilde{C}^{-(k+m)}\tilde{Z}^{-j}x=
$$
$$
=A^{k}X^{*}B^{m-1}(XC)Y^{j-1}\tilde{C}^{-(k+m)}\tilde{Z}^{-j}x=A^{k}(X^{*}X) B^{m-1}Y^{j-1}\tilde{C}^{-(k+m-1)}\tilde{Z}^{-(j-1)}\tilde{Z}^{*}x=
$$
$$
=A^{k} B^{m-1}Y^{j-1}\tilde{C}^{-(k+m-1)}\tilde{Z}^{-(j-1)}\tilde{Z}^{*}x=[\tilde{Z}^{*} x](j-1,k,m-1).
$$
\end{enumerate}
It's easy to see (again using~\eqref{rel}) that the action by $Z,Z^{*}$ is given as
$$
Z [x](j,k,m)=\begin{cases} 0, &\text{if $j,k>0$}\\ [\tilde{Z}x](0,0,m) , & \text{otherwise}\end{cases}
$$
$$
Z^{*} [x](j,k,m)=\begin{cases} 0, &\text{if $j,k>0$}\\ [\tilde{Z}^{*}x](0,0,m) , & \text{otherwise}\end{cases}.
$$
To calculate the action of $Y,$ we break this into the cases $m=0$ and $m>0$. We then have
\begin{enumerate}[(1)]
\item if $m=0,$ then 
$$
Y[x](j,k,0)=YA^{k}Y^{j} \tilde{C}^{-k}\tilde{Z}^{-j}x=A^{k}Y^{j+1}\tilde{C}^{-k}\tilde{Z}^{-(j+1)}\tilde{Z}x=[\tilde{Z} x](j+1,k,0)
$$
\item if $m>0$, then 
$$
YA^{k}B^{m}Y^{j}\tilde{C}^{-(k+m)}\tilde{Z}^{-j}x=A^{k}(YB)B^{m-1}Y^{j}\tilde{C}^{-(k+m)}\tilde{Z}^{-j}x=
$$
$$
=A^{k}(AZ)B^{m-1}Y^{j}\tilde{C}^{-(k+m)}\tilde{Z}^{-j}x=A^{k+1}B^{m-1}ZY^{j}\tilde{C}^{-(k+m)}\tilde{Z}^{-j}x=
$$
$$
=\begin{cases} 0, & \text{if $j>0$}\\ A^{k+1}Z B^{m-1}\tilde{C}^{-(k+m)}\tilde{Z}x=[\tilde{Z} x](0,k+1,m-1),&\text{if $j=0$}.\end{cases}
$$
\end{enumerate}
We have shown that $\oplus_{\mathbf{k}\in \mathbb{N}^{3}}  \mathcal{E}(\mathbf{k})$ is an invariant subspace for $XX^{*}$ and $ZZ^{*}$. As $XX^{*}+YY^{*}+ZZ^{*}=I$ it follows that its invariant under $YY^{*}$ as well.  Moreover, a calculation using ~\eqref{2graph} gives that $$Y^{*}Y=P_{CZ}+P_{CY}+P_{BZ}+P_{AY}=YY^{*}+CC^{*}ZZ^{*}+BB^{*}ZZ^{*},$$ and thus by what we have proven, the subspace $\oplus_{\mathbf{k}\in \mathbb{N}^{3}}  \mathcal{E}(\mathbf{k})$ reduces $Y^{*}Y$ as well. Since we have that $Y$ is a quasinormal partial isometry, if follows that $\oplus_{\mathbf{k}\in \mathbb{N}^{3}}  \mathcal{E}(\mathbf{k})$ is also invariant under $Y^{*}.$ The formulas~\eqref{formulas} now follows from the above calculations.
\end{proof}
\begin{prop}\label{iso0}
We have an isomorphim $\phi:C^{*}(\G)\to SU_{0}(3)$ that intertwines the gauge action $C^{*}(\G)$ with the right-action on $SU_{0}(3).$
\end{prop}
\begin{proof}
From the formulas~\eqref{formulas}, as well as the comments about $C^{*}(\tilde{C},\tilde{Z})$ made before the statement of Propostion~\ref{form}, we have a surjective homomorphism $\phi:C^{*}(\G)\to SU_{0}(3)$. The gauge action $\alpha_{t},$ for $t\in \mathbb{T}^{2},$ on $C^{*}(\G)$ acts on $A,..,Z$ for $t=(t_{1},t_{2})$ as $$\{X,Y,Z\}\mapsto \{t_{1}X,t_{1}Y,t_{1}Z\},$$ $$\{A,B,C\}\mapsto \{t_{2}A,t_{2}B,t_{2}C\}.$$ From~\eqref{betagen}, recall that we had the homomorphism $\beta:\mathbb{T}^{2}\to Aut(SU_{0}(3))$ with action induced by~\eqref{betagen}. It follows that we have $\phi\circ\alpha_{t}=\beta_{t}\circ \phi.$ Clearly $\phi(P_{\sigma})\neq 0$ for all $\sigma\in \Go.$  We then get from Theorem~\ref{gauge} that $\phi$ is faithful, and thus an isomorphism.
\end{proof}

\subsection{Proof that $SU_{0}(3)\cong SU_{q}(3)$}
Recall that we consider the $C^*$-algebra $SU_{0}(3)$ as a closed subalgebra of $ C^{*}(S)^{\otimes  3}\otimes C(\mathbb{T})\otimes C(\mathbb{T}).$ Proposition~\ref{iso0} gave us the result that $SU_{0}(3)\cong C^{*}(\G),$ and thus in order to show that $SU_{q}(3)$ is the graph $C^{*}$-algebra of $\G$, we need to prove that we have an isomorphism $SU_{0}(3)\cong SU_{q}(3).$ In order to do this, we essentially retrace the proof of the q-independence of $SU_{q}(3)$ from~\cite{giselsson}, and extend it to the limit $q\to 0.$ 
\\

In this section, we aim to prove the following result:
\begin{prop}\label{main}
For all $q\in (0,1),$ we have an isomorphism $SU_{0}(3)\cong SU_{q}(3)$ that is equivariant with respect to the right-actions.
\end{prop}
We first need some lemmas:
\begin{lem}\label{indy}
The image of $SU_{q}(3)$ in $$ \left((C^{*}(S)\otimes C^{*}(S)\otimes C(\mathbb{T})\oplus (C(\mathbb{T})\otimes C^{*}(S)\otimes C^{*}(S))\right)\otimes C(\mathbb{T})\otimes C(\mathbb{T})$$ under the representation $$\Lambda_{q}=  \left((\pi_{1}^{(q)}\boxtimes \pi_{2}^{(q)}\boxtimes \tau_{1})\oplus (\tau_{1}\boxtimes \pi_{2}^{(q)}\boxtimes \pi_{1}^{(q)})\right)\boxtimes (\tau_{1}\boxtimes \tau_{2})$$
does not depend on $q.$ Moreover, this $C^{*}$-algebra is generated by the limits $q\to 0$ of the image of the generators $C_{j}^{\omega_{i}},$ for $i=1,2,$ and $j=1,2,3.$
\end{lem}
\begin{proof}
This is verified by a straightforward calculation. Recall that the generators $C_{j}^{\omega_{i}}$ are given by the formulas
$$
\begin{array}{ccc}
C_{j}^{\omega_{1}}=t_{j1}, & C_{j}^{\omega_{2}}=q^{(j-3)}t_{(4-j)3}^{*}.
\end{array}
$$
We calculate the images of these under $\Lambda_{q}$ as
\begin{equation}\label{lamfor}
\begin{array}{ccc}
\Lambda_{q}(C_{1}^{\omega_{1}})^{*}= \left((C_{q} S\otimes I\otimes z)\oplus (z\otimes I\otimes C_{q}S)\right)\otimes z\otimes I, \\
\Lambda_{q}(C_{2}^{\omega_{1}})^{*}=  \left(((-D_{q})\otimes I\otimes z)\oplus (\bar{z}\otimes C_{q}S\otimes( -D_{q}))\right)\otimes z\otimes I,\\
\Lambda_{q}(C_{3}^{\omega_{1}})^{*}= \left((0\otimes 0\otimes 0)\oplus (I\otimes (-D_{q})\otimes (-D_{q}))\right)\otimes z\otimes I,\\
\\
\Lambda_{q}(C_{1}^{\omega_{2}})^{*}=  \left((I\otimes C_{q}S\otimes I)\oplus (I\otimes C_{q}S\otimes I)\right)\otimes I\otimes z,\\
\Lambda_{q}(C_{2}^{\omega_{2}})^{*}=  \left((C_{q}S\otimes D_{q}\otimes I)\oplus (z\otimes D_{q}\otimes I)\right)\otimes I\otimes z,\\
\Lambda_{q}(C_{3}^{\omega_{2}})^{*}= \left(( D_{q}\otimes D_{q}\otimes I)\oplus (0\otimes 0\otimes 0)\right)\otimes I\otimes z.\\
\end{array}
\end{equation}
Notice that all images have norm-continuous limits at $q=0.$ We denote these by $\Lambda_{0}(C_{j}^{\omega_{i}})^{*}$.
It's not hard to see that the two $C^{*}$-algebras $$\mathcal{A}_{1}=C^{*}(\Lambda_{q}(C_{1}^{\omega_{1}})^{*},\Lambda_{q}(C_{2}^{\omega_{1}})^{*},\Lambda_{q}(C_{3}^{\omega_{1}})^{*}),$$  $$\mathcal{A}_{2}=C^{*}(\Lambda_{q}(C_{1}^{\omega_{2}})^{*},\Lambda_{q}(C_{2}^{\omega_{2}})^{*},\Lambda_{q}(C_{3}^{\omega_{2}})^{*})$$ are independent on $q$ and generated by the limits $q\to 0.$ 
\\

We prove this for $\mathcal{A}_{1}.$ The case of $\mathcal{A}_{2}$ is treated similarly. 
Take the limits $q\to 0$ to get
\begin{equation}\label{formim}
\begin{array}{ccccc}
\Lambda_{q}(C_{1}^{\omega_{1}})^{*}= \left(( S\otimes I\otimes z)\oplus (z\otimes I\otimes S)\right)\otimes z\otimes I, \\
\Lambda_{0}(C_{2}^{\omega_{1}})^{*}= \left(((-P)\otimes I\otimes z)\oplus (\bar{z}\otimes S\otimes( -P))\right)\otimes z\otimes I,\\
\Lambda_{0}(C_{3}^{\omega_{1}})^{*}=  \left((0\otimes 0\otimes 0)\oplus (I\otimes (-P)\otimes (-P))\right)\otimes z\otimes I.
\end{array}
\end{equation}
First we show that $\Lambda_{0}(C_{i}^{\omega_{1}})^{*}\in \mathcal{A}_{1}$ for $i=1,2,3.$ Clearly, we have that $\Lambda_{q}(C_{1}^{\omega_{1}})\Lambda_{q}(C_{1}^{\omega_{1}})^{*}$ is invertible. It follows that
$$
\Lambda_{q}(C_{1}^{\omega_{1}})^{*}(\Lambda_{q}(C_{1}^{\omega_{1}})\Lambda_{q}(C_{1}^{\omega_{1}})^{*})^{-\frac{1}{2}}= \left((S\otimes I\otimes z)\oplus(z\otimes I\otimes S)\right)\otimes \bar{z}\otimes I=
$$
$$
=\Lambda_{0}(C_{1}^{\omega_{1}})^{*}\in \mathcal{A}_{1}.
$$  

From this, we get that also
$$
\Lambda_{0}(C_{1}^{\omega_{1}})\Lambda_{0}(C_{1}^{\omega_{1}})^{*}-\Lambda_{0}(C_{1}^{\omega_{1}})^{*}\Lambda_{0}(C_{1}^{\omega_{1}})= \left((P\otimes I\otimes I)\oplus(I\otimes I\otimes P)\right)\otimes I\otimes I\in \mathcal{A}_{1}.
$$
Thus
$$
\left( \left((P\otimes I\otimes I)\oplus(I\otimes I\otimes P)\right)\otimes I\otimes I\right)\Lambda_{q}(C_{2}^{\omega_{1}})^{*}=
$$
\begin{equation}\label{help}
=\left(((-P)\otimes I\otimes z)\oplus (\bar{z}\otimes C_{q}S\otimes (-P))\right)\otimes \bar{z}\otimes I\in \mathcal{A}_{1}.
\end{equation}
It's not hard to see that $0$ is an isolated point in the spectrum of the absolute value of~\eqref{help}. Thus the partial isometry in its polar decomposition lies in $\mathcal{A}_{1},$ so that
$$
 \left((P\otimes I\otimes z)\oplus (\bar{z}\otimes S\otimes P)\right)\otimes\bar{z}\otimes I=\Lambda_{0}(C_{2}^{\omega_{1}})^{*}\in \mathcal{A}_{1}
$$
From~\eqref{lamfor}, it's also clear that $1$ is an isolated point in the spectrum of $\Lambda_{q}(C_{3}^{\omega_{1}})\Lambda_{q}(C_{3}^{\omega_{1}})^{*}$, it follows from the spectral theorem that $  \left((0\otimes 0\otimes 0)\oplus (I\otimes P\otimes P)\right)\otimes I\otimes I\in \mathcal{A}_{1},$ so that we have 
$$
  \left((0\otimes 0\otimes 0)\oplus (I\otimes P\otimes P)\right)\Lambda_{q}(C_{3}^{\omega_{1}})^{*}=
$$
$$
=  \left((0\otimes 0\otimes 0)\oplus (I\otimes (-P)\otimes (-P))\right)\otimes \bar{z}\otimes I=\Lambda_{0}(C_{3}^{\omega_{1}})^{*}\in \mathcal{A}_{1}.
$$
This finishes the proof that $\Lambda_{0}(C_{i}^{\omega_{1}})^{*}\in \mathcal{A}_{1}$ for $i=1,2,3.$ 
\\
That these elements also generates $\mathcal{A}_{1}$ follows from the easily verified norm-convergent formulas
$$
\Lambda_{q}(C_{1}^{\omega_{1}})^{*}=\Lambda_{0}(C_{1}^{\omega_{1}})^{*}\left(\Lambda_{q}(C_{1}^{\omega_{1}})\Lambda_{q}(C_{1}^{\omega_{1}})^{*}\right)^{\frac{1}{2}}=
$$
$$
\Lambda_{0}(C_{1}^{\omega_{1}})^{*}\left((1-q^{2})\sum_{k=0}^{\infty}q^{2k}(\Lambda_{0}(C_{1}^{\omega_{1}})^{*})^{k}(\Lambda_{0}(C_{1}^{\omega_{1}}))^{k}\right)^{\frac{1}{2}},
$$
$$
\Lambda_{q}(C_{2}^{\omega_{1}})^{*}=
$$
$$
\sum_{k=0}^{\infty}q^{k}\Lambda_{0}(C_{1}^{\omega_{1}})^{*k}\left((1-q^{2})\sum_{j=1}^{\infty}q^{2(j-1)}\Lambda_{0}(C_{2}^{\omega_{1}})^{*j}\Lambda_{0}(C_{2}^{\omega_{1}})^{j}\right)^{\frac{1}{2}}\Lambda_{0}(C_{2}^{\omega_{1}})^{*}\Lambda_{0}(C_{1}^{\omega_{1}})^{k},
$$
$$
\Lambda_{q}(C_{3}^{\omega_{1}})^{*}=
$$
$$
\sum_{k,j=0}^{\infty}q^{k+j}\left((\Lambda_{0}(C_{1}^{\omega_{1}})^{*})^{k
}(\Lambda_{0}(C_{2}^{\omega_{1}})^{*})^{j}\right)\Lambda_{0}(C_{3}^{\omega_{1}})^{*}\left((\Lambda_{0}(C_{2}^{\omega_{1}}))^{j}(\Lambda_{0}(C_{1}^{\omega_{1}}))^{k}\right).
$$
\end{proof}
By~\cite{KorS}, we have that $SU_{q}(3)$ is a type $\mathrm{I}$ $C^{*}$-algebra. Thus the $C^{*}$-algebra from Lemma~\ref{indy} is type $\mathrm{I}$ as well, by virtue of being the image of $SU_{q}(3)$ under a homomorphism. 
\\

One of the central idea in G. Nagy paper~\cite{gnagy} (e.g. to show that the $C^{*}$-algebras $SU_{q}(3)$ are all isomorphic) was to use the following lifting result, based on the theory developed in~\cite{ppv}, to extend the trivial isomorphisms from Lemma~\ref{indy}, to isomorphisms between the full $C^{*}$-algebras.  
\begin{lem}[{Lemma $2$ in~\cite{gnagy}}]\label{gnagy}
Let $\mathrm{H}$ be a separable Hilbert space, let $\mathcal{K}$ be the space of compact operators on $\mathrm{H},$ let $Q(\mathrm{H})=\mathcal{B}(\mathrm{H})/\mathcal{K}$ be the Calkin algebra and $p: \mathcal{B}(\mathrm{H})\to Q(\mathrm{H})$ the quotient map. Suppose $A$ is a fixed separable $C^{*}$-algebra of type $\mathrm{I}$ and $\psi_{q}:A\to Q(\mathrm{H}),$ $q\in [0,1]$ is a point-norm continuous family of injective $*$-homomorphisms. Denote  $$\mathfrak{A}_{q}:=\psi_{q}(A):$$ 
$$
M_{q}:=p^{-1}(\mathfrak{A}_{q}).
$$
Then there exists a family of injective $*$-homomorphisms $\Psi_{q}:M_{0}\to \mathcal{B}(\mathrm{H}),$ $q\in [0,1]$ with the following properties
\begin{enumerate}[(a)]
\item $\Psi_{q}(M_{0})=M_{q}$ for $q\in [0,1]$ and $\Psi_{0}=\id_{M_{0}},$
\item the family $\Psi_{q}:M_{0}\to \mathcal{B}(\mathrm{H}),$ $q\in[0,1]$ is point-norm continuous,
\item for every $q\in [0,1],$ the diagram
\begin{equation}\label{com1}
\begin{tikzpicture}[>=latex',node distance = 1.5cm]'
\node at (0,0) (A) {\text{$M_{0}$}};
\node at (2,0) (B) {\text{$M_{q}$}};
\node at (0,-2) (C) {\text{$\mathfrak{A}_{0}$}};
\node at (2,-2) (D) {\text{$\mathfrak{A}_{q}$}};
\draw[->](A) to node [above]{\text{$\Psi_{q}$}} (B);
\draw[->](A) to node [left]{\text{$p$}} (C);
\draw[->](B) to node [right]{\text{$p$}} (D);
\draw[->](C) to node [below]{\text{$\psi_{q}\circ \psi_{0}^{-1}$}} (D);
\end{tikzpicture}
\end{equation}
is commutative.
\end{enumerate}
\end{lem}
By using similar arguments (or rather half the arguments) as Lemma 2 in ~\cite{giselsson}, one can extend this Lemma to the half-closed interval $[0,1).$ We remark that this extension is not really needed for the purpose here, as we are mostly interested in establish an isomorphism between $SU_{0}(3)$ and $SU_{q}(3),$ and thus need to only consider the interval $[0,q].$
\begin{lem}\label{compacts}
Let $\mathcal{K}\subseteq C^{*}(S)^{\otimes 3}$ denotes the $C^{*}$-algebra of compact operators. We have $ \mathcal{K}\otimes C(\mathbb{T})\otimes C(\mathbb{T})\subseteq SU_{0}(3)$.
\end{lem}
\begin{proof}
This is a calculation done using the operators~\eqref{a}-\eqref{z}. We have $(\hat{C}^{*} \hat{C})(\hat{Z}^{*}\hat{Z} )= P\otimes P\otimes P\otimes I\otimes I\in SU_{0}(3).$ Thus also 
$$
\hat{C}( P\otimes P\otimes P\otimes I\otimes I)= P\otimes P\otimes P\otimes z\otimes I\in SU_{0}(3),
$$
$$
\hat{Z}( P\otimes P\otimes P\otimes I\otimes I)= P\otimes P\otimes P\otimes I\otimes z\in SU_{0}(3).
$$
We then get from Stone–Weierstrass theorem that $ P\otimes P\otimes P\otimes C(\mathbb{T})\otimes C(\mathbb{T})\subseteq SU_{0}(3).$ Assume that we have the usual inclusion $C(\mathbb{T})\subseteq \mathcal{B}(L^{2}(\mathbb{T})).$ A simple calculation shows that for all $(j,k,m)\in \mathbb{N}^{3}$, we have
$$
(\hat{A}^{k}\hat{B}^{m}\hat{Y}^{j}\hat{C}^{-(m+k)}\hat{Z}^{-j})( x_{0}\otimes x_{0}\otimes x_{0}\otimes 1\otimes 1)= x_{j}\otimes x_{k}\otimes x_{m}\otimes 1\otimes 1.
$$
From this, it follows easily that $ \mathcal{K}\otimes C(\mathbb{T})\otimes C(\mathbb{T})\subseteq SU_{0}(3).$
\end{proof}
\begin{proof}[Proof of Theorem~\ref{main}]
We have the $SU_{q}(3)$-representation  $$\Pi_{q}=\pi_{1}^{(q)}\boxtimes \pi_{2}^{(q)}\boxtimes \pi_{1}^{(q)}:SU_{q}(3)\to C^{*}(S)^{\otimes 3}\subseteq \mathcal{B}(\ell^{2}(\mathbb{N})^{\otimes 3}),$$ where $\pi_i:SU_{q}(3)\to C^{*}(S)$ is as~\eqref{pii}. By~\cite{KorS}, this representation is irreducible. As $SU_{q}(3)$ is type $\mathrm{I},$ we have 
$$\mathcal{K}\subseteq \im \Pi_{q}\subseteq C^{*}(S)^{\otimes 3}$$ (where again $\mathcal{K}$ denotes the compact operators). Consider now the homomorphism defined as the composition
\begin{equation}\label{compo}
SU_{q}(3)\overset{\Pi_{q}}{\to}\mathcal{B}(\ell^{2}(\mathbb{N})^{\otimes 3}) \overset{\sim \mathcal{K}}{\to}\mathcal{B}(\ell^{2}(\mathbb{N})^{\otimes 3})/\mathcal{K}=\mathcal{Q}(\ell^{2}(\mathbb{N})^{\otimes 3}).
\end{equation}
We can deduce from the $SU_{q}(3)$-representation theory (Theorem 4.1 (ii) in~\cite{nt}) that the kernel of the representation~\eqref{compo} is the same as for the representation
$$
\Phi_{q}=(\pi_{1}^{(q)}\boxtimes \pi_{2}^{(q)}\boxtimes \tau_{1})\oplus (\tau_{1}\boxtimes \pi_{2}^{(q)}\boxtimes \pi_{1}^{(q)}).
$$
By evatuating the $C^{*}$-algebra from Lemma~\ref{indy} at the point $(1,1)\in \mathbb{T}^{2}$ at the factor $C(\mathbb{T})\otimes C(\mathbb{T})$, one sees that the image of $\Phi_{q}$ is actually independent of $q$, and that it's generated as a $C^{*}$-algebra by the limits $\lim_{q\to 0}\Phi_{q}(C_{j}^{\omega_{i}}).$
\\

We denote $A:=\im\, \Phi_{q}$. By comparison of kernels, it follows from the representation theory of $SU_{q}(3)$ that we have an injective homomorphism $\varphi_{q}:A\to \mathcal{Q}(\ell^{2}(\mathbb{N})^{\otimes 3})$ such that the following diagram commutes
\begin{equation}\label{maps}
\begin{tikzpicture}[>=latex',node distance = 1.5cm]'
\node at (0,-2) (A) {\text{$\mathcal{B}(\ell^{2}(\mathbb{N})^{\otimes 3})$}};
\node at (0,0) (B) {\text{$C(SU_{3})_{q}$}};
\node at (4,-2) (C) {\text{$\mathcal{Q}(\ell^{2}(\mathbb{N})^{\otimes 3})$}};
\node at (4,0) (D) {\text{$A$}};
\draw[->](A) to node [below]{\text{$\sim \mathcal{K}$}} (C);
\draw[->](B) to node [left]{\text{$\Pi_{q}$}} (A);
\draw[->](B) to node [above]{\text{$\Phi_{q}$}} (D);
\draw[->](D) to node [right]{\text{$\varphi_{q}$}} (C);
\end{tikzpicture}
\end{equation}
Let $F(\mathbf{t}(q))$ be a fixed non-commutative polynomial in the set $\{C_{j}^{\omega_{i}},\,(C_{j}^{\omega_{i}})^{*}\,|\,j=1,2,\, i=1,2,3\}$. Such elements are dense in $SU_{q}(3)$ for $q\in [0,1)$. As the images of the generators $C_{j}^{\omega_{i}}$ under $\Phi_{q}$ and $\Pi_{q}$ depends norm-continuously on $q\in [0,1)$, it follows that so does $\Phi_{q}(F(\mathbf{t}(q)))$ and $\Pi_{q}(F(\mathbf{t}(q))).$ In particular, we have norm-limits 
$$\underset{q\to 0}{\lim}\,\Phi_{q}(F(\mathbf{t}(q)))=\Phi_{0}(F(\mathbf{t}(0))),$$  $$\underset{q\to 0}{\lim}\,\Pi_{q}(F(\mathbf{t}(q)))=\Pi_{0}(F(\mathbf{t}(0))).$$ We claim that $\varphi_{q}$ depends point-norm continuously on $q\in [0,1)$. To prove continuity at $q\in (0,1),$ we consider elements in $A$ of the form $\Phi_{q}(F(\mathbf{t}(q)))$. We have
$$
||\varphi_{q}(\Phi_{q}(F(\mathbf{t}(q))))-\varphi_{q+\epsilon}(\Phi_{q}(F(\mathbf{t}(q))))||\leq
$$
$$
\leq || \varphi_{q}(\Phi_{q}(F(\mathbf{t}(q))))-\varphi_{q+\epsilon}(\Phi_{q+\epsilon}(F(\mathbf{t}(q+\epsilon))))||+
$$
$$
+|| \varphi_{q+\epsilon}(\Phi_{q+\epsilon}(F(\mathbf{t}(q+\epsilon))))-\varphi_{q+\epsilon}(\Phi_{q}(F(\mathbf{t}(q))))||\leq
$$
$$
\leq || \Pi_{q}(F(\mathbf{t}(q)))-\Pi_{q+\epsilon}(F(\mathbf{t}(q+\epsilon)))||+|| \Phi_{q+\epsilon}(F(\mathbf{t}(q+\epsilon)))-\Phi_{q}(F(\mathbf{t}(q)))||,
$$
and by the above comments, both terms in the last expression $\to 0$ as $|\epsilon| \to 0.$ The claim follows as these elements are dense in $A$. We show the existence of the point-norm limit $\lim_{q\to 0}\varphi_{q},$ and that $\varphi_{0}$ maps $\Phi_{0}(F(\mathbf{t}(0)))$ to $\Pi_{0}(F(\mathbf{t}(0)))+\mathcal{K}$. For $\epsilon>0$, we have
$$
||\Pi_{0}(F(\mathbf{t}(0)))+\mathcal{K}-\varphi_{\epsilon}(\Phi_{0}(F(\mathbf{t}(0))))||\leq
$$
$$
||\Pi_{0}(F(\mathbf{t}(0)))+\mathcal{K}-\varphi_{\epsilon}(\Phi_{\epsilon}(F(\mathbf{t}(\epsilon))))||+||\varphi_{\epsilon}(\Phi_{\epsilon}(F(\mathbf{t}(\epsilon))))-\varphi_{\epsilon}(\Phi_{0}(F(\mathbf{t}(0))))||\leq
$$
$$
||\Pi_{0}(F(\mathbf{t}(0)))-\Pi_{\epsilon}(F(\mathbf{t}(\epsilon)))||+||\Phi_{\epsilon}(F(\mathbf{t}(\epsilon)))-\Phi_{0}(F(\mathbf{t}(0)))||,
$$
and both these last terms converges $\to 0$ as $\epsilon \to 0.$ Again, by the density of $\Phi_{0}(F(\mathbf{t}(0)))$, the point-norm limit $\varphi_{0}:A\to \mathcal{Q}(\ell^{2}(\mathbb{N})^{\otimes 3})$ exists and is an injective homomorphim (as it's a point-norm limit of injective homomorphisms). We can thus apply Lemma~\ref{gnagy} to get isomorphisms $$\Gamma_{q}:p^{-1}(\varphi_{0}(A))\to p^{-1}(\varphi_{q}(A))=\Phi_{q}(C(SU_{3})_{q}).$$ Let us denote $M:=p^{-1}(\varphi_{0}(A)).$ By Lemma~\ref{compacts}, evaluating the $C(\mathbb{T})\otimes C(\mathbb{T})$-factors in $SU_{0}(3)$ at a point shows that the $C^{*}$-algebra generated by elements $\Pi_{0}(F(\mathbf{t}(0)))$ contains $\mathcal{K}$. Thus it coincides with $M$. In particular we have that $$SU_{0}(3)\subseteq   M\otimes C(\mathbb{T})\otimes C(\mathbb{T})\subseteq  C^{*}(S)^{\otimes 3}\otimes C(\mathbb{T})\otimes C(\mathbb{T}).$$ As 
$$ (\pi^{q}_{1}\boxtimes \pi^{q}_{2}\boxtimes \pi^{q}_{1})\boxtimes (\tau_{1}\boxtimes\tau_{2})= \Pi_{q}=\Xi_{q}\boxtimes (\tau_{1}\boxtimes\tau_{2})$$ 
gives a faithful representation of $SU_{q}(3),$ we can consider this $C^{*}$-algebra as a sub-$C^{*}$-algebra $$SU_{q}(3)\subseteq  \Pi_{q}(SU_{q}(3))\otimes C(\mathbb{T})\otimes C(\mathbb{T})\subseteq  C^{*}(S)^{\otimes 3}\otimes C(\mathbb{T})\otimes C(\mathbb{T}).$$ We now show that $ \Gamma_{q}\otimes \iota\otimes\iota$ restricts to an isomorphism
\begin{equation}\label{mapsto}
\begin{array}{ccc}
\Gamma_{q}\otimes \iota\otimes\iota:SU_{0}(3)\to SU_{q}(3), & \text{for $q\in (0,1)$.}
\end{array}
\end{equation}
To do this, note that we have the ideal $ \mathcal{K}\otimes C(\mathbb{T})\otimes C(\mathbb{T})\subseteq SU_{q}(3)$ for $q\in [0,1)$. For $q\in (0,1)$ this is standards theory (Lemma 16 in~\cite{giselsson}), and for $q=0$ this is Lemma~\ref{compacts}. By the diagram~\eqref{com1}, the isomorphisms $\Gamma_{q}$ fixes $\mathcal{K}$, so that~\eqref{mapsto} fixes $ \mathcal{K}\otimes C(\mathbb{T})\otimes C(\mathbb{T}).$ Thus to show that~\eqref{mapsto} is an isomorphism, it's enough to check this modulo this ideal. By~\eqref{com1}, this is the same as showing that $ \varphi_{q}\circ \varphi_{0}^{-1}\otimes \iota\otimes\iota$ restricts to a isomorphism
\begin{equation}\label{both}
( \varphi_{q}\circ \varphi_{0}^{-1}\otimes \iota\otimes\iota):SU_{0}(3)/\mathcal{K}\otimes C(\mathbb{T})\otimes C(\mathbb{T})\to SU_{q}(3)/ \mathcal{K}\otimes C(\mathbb{T})\otimes C(\mathbb{T}).
\end{equation}
Notice that for $q\in[0,1),$ we have $$SU_{q}(3)/ \mathcal{K}\otimes C(\mathbb{T})\otimes C(\mathbb{T})\subseteq  \varphi_{q}(A)\otimes C(\mathbb{T})\otimes C(\mathbb{T}),$$
so that~\eqref{both} can be reduced to the question if the following equality holds
\begin{equation}\label{sides}
( \varphi_{0}^{-1}\otimes \iota\otimes \iota)(SU_{0}(3)/ \mathcal{K}\otimes C(\mathbb{T})\otimes C(\mathbb{T}))=( \varphi_{q}^{-1}\otimes\iota\otimes \iota)(SU_{q}(3)/ \mathcal{K}\otimes C(\mathbb{T})\otimes C(\mathbb{T}))
\end{equation}
as sub-$C^{*}$-algebras of $A\otimes C(\mathbb{T})\otimes C(\mathbb{T}).$ Due to~\eqref{maps}, the question of equality in~\eqref{sides} is the same as whether we have independence of $q$ of the images of
$$
 \Phi_{q}\boxtimes \tau_{1}\boxtimes \tau_{2}= \left((\pi_{1}^{(q)}\boxtimes \pi_{2}^{(q)}\boxtimes \tau_{1})\oplus (\tau_{1}\boxtimes \pi_{2}^{(q)}\boxtimes \pi_{1}^{(q)})\right)\boxtimes\tau_{1}\boxtimes \tau_{2}.
$$
However, this was proven in Lemma~\ref{indy}. Thus~\eqref{sides} holds, and hence $\Gamma_{q}\otimes \iota\otimes\iota$ induces an isomorphism $SU_{0}(3)\to SU_{q}(3).$ By~\eqref{beta}, we have for all $t\in \mathbb{T}$ that $$(\Gamma_{q}\otimes \iota\otimes\iota)\circ\beta_{t}=\beta_{t}\circ (\Gamma_{q}\otimes \iota\otimes\iota).$$
\end{proof}
\begin{rem}
Since $\Gamma_{q}$ fixes the compacts, it can be written as $\Gamma_{q}(a)=U_{q}^{*}a U_{q}$ for some unitary $U_{q}\in \mathcal{B}(\ell^{2}(\mathbb{N})^{\otimes 3})$ (Corollary 1.10 in~\cite{dav}).
\end{rem}
By combining Proposition~\ref{iso0} with Proposition~\ref{main}, we have proven the main result here:
\begin{thm}
For every $q\in (0,1),$ we have an isomorphism $\phi_{q}:C^{*}(\G)\to SU_{q}(3),$ intertwining the gauge action on $C^{*}(\G)$ with the right-action on $SU_{q}(3)$.
\end{thm}
\begin{proof}
Take $\phi_{q}=(\Gamma_{q}\otimes\iota\otimes\iota)\circ\phi,$ where $\phi$ is the isomorphism from Proposition~\ref{iso0}, and $\Gamma_{q}\otimes\iota\otimes\iota$ is the isomorphism from the proof of Proposition~\ref{main}.
\end{proof}

\end{document}